\documentclass[oneside,leqno,12pt]{amsart}

\usepackage{fullpage}
\usepackage{amsfonts}
\usepackage{amsmath}
\usepackage{amsthm}
\usepackage{dsfont}
\usepackage{amsthm}
\usepackage{amssymb} 
\usepackage{amsxtra}     
\usepackage{epsfig}
\usepackage{verbatim}
\usepackage{color}
\usepackage{enumitem}
\usepackage{url}

\newtheorem{theorem}{Theorem}
\newtheorem{prop}[theorem]{Proposition}

\newtheorem{lemma}[theorem]{Lemma}

\numberwithin  {equation}{section}

\newcommand{\R }{\mathbb{R}}
\newcommand{\NN }{\mathbb{N}}

\newcommand{\N}{\mathcal N}

\newcommand{\ds}{\displaystyle}
\newcommand{\de}{\delta}
\newcommand{\al}{\alpha}
\begin{document}
	
	\large
\title {Large versus bounded solutions to sublinear elliptic problems}
\author{ E. Damek,\\ Institute of Mathematics, Wroclaw University,\\  50-384 Wroclaw,
pl. Grunwaldzki 2/4, Poland\\   
 edamek@math.uni.wroc.pl\\ \\
 Zeineb Ghardallou,\\ Department of Mathematical Analysis and Applications, \\ University Tunis El Manar\\ LR11ES11, 2092 El Manar1, Tunis, Tunisia
\\
zeineb.ghardallou@ipeit.rnu.tn\\}
\maketitle
\footnotetext{ \noindent E. Damek is the corresponding author, +48698599950, orcid.org/0000-0003-2847-6150. 

\noindent The first author was supported by the NCN Grant UMO-2014/15/B/ST1/00060.

\noindent Z. Ghardallou orcid.org/0000-0002-9668-0228.

\noindent The authors declare that they have no conflict of interests.


\noindent
\textbf{Mathematic Subject Classification (2010):} 31C05; 31D05; 35J08; 35J61.
}




\begin{abstract} Let $L $ be a second order elliptic operator with smooth coefficients defined on a domain $\Omega \subset \mathbb{R}^d$ (possibly unbounded), $d\geq 3$. 
	We study nonnegative continuous solutions $u$ to the equation $L u(x) - \varphi (x, u(x))=0$ on $\Omega $,
	where $\varphi $ is in the Kato class with respect to the first variable and it grows sublinearly with respect to the second variable. Under fairly general assumptions we prove that if there is a bounded non zero solution {then} there is no large solution. 
\end{abstract}

\section{Introduction}
{Let $L $ be a second order elliptic operator 
\begin{equation}\label{operator}
L=\sum _{i,j=1}^da_{ij}(x)\partial _{x_i}\partial _{x_j}
+\sum _{i=1}^db_{i}(x)\partial _{x_i}\end{equation}
with smooth coefficients $a_{ij}$, $b_i$ defined on a domain $\Omega \subset \mathbb{R}^d$, $d\geq 3$ \footnote{By a domain we always mean a set that is open and connected.}. No conditions are put on the behavior of $a_{ij}, b_j$ near the boundary of $\partial \Omega $.}
We study nonnegative continuous functions $u$ such that
\begin{equation}\label{problem}
L u(x) - \varphi (x, u(x))=0, \ \mbox{on}\  \Omega ,
\end{equation}
in the sense of distributions, where $\varphi :\Omega  \times [0,\infty )\to [0,\infty )$ grows
sublinearly with respect to the second variable. Such $u$ will be later called {\it solutions}.
A solution $u$ to  \eqref{problem} is called {\it large } if $u(x)\to \infty $ when $x\to \partial
\Omega $ or $\|x\| \to \infty $.

{Large solutions i.e. the boundary blow-up problems are of considerable interest due to its several scientific applications in different fields. Such problems arise in the study of Riemannian geometry \cite{Cheng-Ni}, non-Newtonian fluids \cite{Astrita-Marrucci}, the subsonic motion of a gas \cite{Pohzaev} and the electric potential in some bodies \cite{Lazer-Mckenna}}.
	
	 We prove that under fairly general conditions bounded and large
solutions cannot exist at the same time. Classical examples, the reader may have in mind are
\begin{equation}\label{special}
\Delta u - p(x)u^{\gamma }=0\quad \mbox{with}\ 0<\gamma \leq 1 \mbox{ and } p\in \mathcal{ L}_{loc}^\infty, \end{equation}
where $\Delta$ is the Laplace operator on $\mathbb{R} ^d $, but we go far beyond that. 
Not only the operator may be more general but 
the special form of the nonlinearity in \eqref{special} may be replaced by $\varphi (x,t)$ satisfying  

\medskip
\begin{description}
	\item[{($SH_1$)}] {There exists a function $p\in \mathcal{K}_d^{loc}(\Omega)$ locally in the Kato class} such that for every $t\geq 0, x\in \Omega$,  $\varphi(x,t)\leq p(x)(t+1)$.
	\item[$(H_2)$] For every $x\in \Omega$,  $t\mapsto \varphi(x, t )$ is   continuous
	nondecreasing on $[0,+\infty)$. 
	
	  \item[{$(H_3)$}]{$\varphi(x,t)=0$ for every $x\in \Omega$ and $t\leq 0$}.

\end{description}

\medskip
 We recall that a Borel measurable function $\psi$ on $\Omega$ is locally in the Kato class
	in $\Omega$ if $$
	\lim\limits_{\alpha\to 0}\sup_{x\in D}\int_{D\cap(|x-y|\leq
		\alpha)}\frac{|\psi(y)|}{|x-y|^{d-2}}\,dy=0$$ for every open bounded set $D$, $\bar D\subset \Omega $. {$(H_1)$ makes $\varphi$ locally integrable against against the Green function\footnote{ See \eqref{regularityGOmega}, \eqref{delta}, \eqref{potential} for the definition of $G_\Omega$. } for $L$ which plays an important role in our approach. $(H_3)$ is a technical extension of $\varphi$ to $(-\infty,0)$ needed as a tool.}
{For a part of results we replace $ (SH_1)$ by a weaker condition $(H_1)$:

\medskip
\begin{description}
	\item[$(H_1)$]  For every $t\in [0,+\infty )$, $x\mapsto \varphi(x,t)\in \mathcal{ K}_d^{loc}(\Omega).$
\end{description}}

\medskip
Applying methods of potential theory we obtain the following result.

\begin{theorem}\label{existence}
{Assume that $\Omega $ is Greenian for $L$ \footnote{See Section 
		\ref{largesol1}
		for the definition, more precisely, \eqref{delta}, \eqref{potential}.}.
	Suppose that $\varphi (x,t)=p(x)\psi (t)$ satisfies $(SH_1)$, $(H_2)$, $(H_3)$ and there exists a nonnegative nontrivial bounded solution to \eqref{problem}. Then there is no large solution to \eqref{problem}.}
\end{theorem}

Theorem \ref{existence} improves considerably a similar result of El Mabrouk and Hansen
\cite{ElmabroukHansen} for $L$ being the Laplace operator $\Delta$ on $\R ^d$, $\varphi
(x,t)=p(x)\psi (t)$, $p\in \mathcal{L}^{\infty }_{loc} (\mathbb{R}^d)$ and $\psi
(t)=t^{\gamma}$,  $0<\gamma<1$. It is proved in Section \ref{largesol1}. 

In fact, we prove a few more general statements than Theorem
\ref{existence} but they are a little bit more technical to formulate and so we refer to Theorem \ref{non-existence-result-second-version} in 
Section \ref{largesol0}. Generally, we do not assume that $\varphi $ has product form and, in
particular, we characterize a class of functions $p(x)$ in $(SH_1)$ for
which there are bounded solutions and do not exist large solutions to \eqref{problem}, see Theorem 
\ref{characterizationc} in Section \ref{largesol1} .

Besides the theorem due to El Mabrouk and Hansen \cite{ElmabroukHansen} there are other results indicating that the equation $\Delta u- p(x)u^{\gamma }=0$ or, more generally, $\Delta u- p(x)\psi (u)=0$ can not have bounded and large solutions at the same time \cite{L}, \cite{L1}, \cite{LW}. We prove such a statement in a considerable generality:
\begin{itemize}
	\item $L$ is an elliptic operator { \eqref{operator}}
	\item $\Omega $ is Greenian for $L$, generally unbounded
	\item the nonlinearity is assumed to have only a sublinear growth, no concavity with respect to the second variable and no product form of $\varphi $ is required.
\end{itemize}
Our main strategy adopted from \cite{ELMabrouk2004} and \cite{ElmabroukHansen} is to relate solutions of \eqref{problem} to $L$-harmonic functions and to make an extensive use of potential theory. We rely on the results of \cite{Ghardallou1} and \cite{Ghardallou2} where such approach was developed.

Existence of large solutions for the equation
\begin{equation*}
\Delta u =p(x)f(u)
\end{equation*}
was studied under more regularity: $p$ H\" older continuous and $f$ Lipschitz (not necessarily monotone), \cite{LPW}
\footnote{More generally, $\Delta u =p(x)f(u)+q(x)g(u)$, $p,q$ H\" older continuous, \cite{LM}.}
or on the whole of $\mathbb{R} ^d$, \cite{YZ}. 
In our approach very little regularity is involved but mononicity of $\varphi $ with respect of $t$ is essential.
If $\varphi $ is not of the product form  but the following condition is  satisfied

\medskip
{ \begin{description}
	\item[$(H_4)$] For every $x\in \Omega$,  $t\mapsto \varphi(x, t )$ is concave on $[0,+\infty )$.
\end{description}}
then we have

\begin{theorem}\label{existence1}
	Suppose that $(H_1)-(H_4)$ hold and that there is a bounded solution to 
	$$
	L u(x) - \varphi (x, u(x))=0.
	$$
	Then there is no large solution.
\end{theorem}

 Theorem \ref{existence1} follows directly from Theorem \ref{non-existence-result-second-version}.
Our strategy for the proof of Theorem \ref{existence} is to construct a function $\varphi _1\geq \varphi $ satisfying $(SH_1)$, $(H_2)-(H_4)$ and to use the statement for $\varphi _1$.\footnote{ The main difficulty is to guarantee  that $\varphi _1(x,0)=0$, see Section \ref{domin}. } To make use of both equations, for $\varphi $ and $\varphi _1$, we need a criterion for existence of bounded solutions to \eqref{problem}, see Theorem \ref{sufficient-and-necessary-condition-existence-solution}. The latter proved in such generality, is itself interesting.

Semilinear problems $\Delta u + g(x,u)=0$ have been extensively studied under variety hypotheses on
$g$ and various questions have been asked. $g$ is not necessarily monotone or negative but there
are often other restrictive assumptions like more regularity of $g$ or the product form. The
problem is usually considered either in bounded domains or in $\Omega =\mathbb{R} ^d$ \cite{BanMar}, \cite{Diaz}, \cite{Dinu}, \cite{ELMabrouk2006}, \cite{Fen}, \cite{Goncalves}, \cite{Guo}, \cite{LS}, \cite{Ahmed}, \cite{Shi}, \cite{Y}, \cite{ZCh}, \cite{Zhang}. Finally, there
not many results for general elliptic operators and if so, the same restrictions apply
\cite{Crandal}, \cite{HMV}, \cite{Sat}, \cite{Stuart}. Clearly stronger regularity of $g$ or $\Omega $ is used to obtain conclusions other than the one we are interested in.

\section{Large solutions to $L u-\varphi (\cdot ,u)=0$ under $(H_1)$, $(H_2)$, $(H_3)$ }\label{largesol0}

{ In this section we replace $ (SH_1)$ by $(H_1)$ which is weaker. Our aim is to prove that under fairly general assumptions bounded and large solutions to \eqref{problem} cannot occur at the same time.\footnote{In Theorem \ref{non-existence-result-second-version} and all the statements of this section $L$ may be slightly more general: a nonpositive zero order term is allowed.}}

\begin{theorem}\label{non-existence-result-second-version}
	
	\vspace{0.2cm}
	 Let $\Omega$ be a domain and $\varphi,\varphi_1:\Omega \times [0,\infty [ \to [0,\infty[ $	satisfies $(H_1), (H_2), (H_3)$.  Assume that $\varphi\leq \varphi_1$ and $\varphi_1$ is concave with
	respect to the second variable. If the equation $Lu=\varphi_1(\cdot,u)$ has a nontrivial nonnegative
	bounded solution in $\Omega$ then $Lu=\varphi (\cdot,u)$ does not have a large solution in $\Omega$.

\end{theorem}
Theorem \ref{non-existence-result-second-version} gives, in particular, the most general conditions for $\Delta $ implying non existence of a bounded and a large solution at the same time. Compare with Theorem 3.1 in \cite{ElmabroukHansen}, where the statement was proved for $\varphi (x,u)=p(x)u^{\gamma }$, $p\in\mathcal{ L} ^{\infty }_{loc}(\Omega )$.

Applying Theorem \ref{non-existence-result-second-version} to $\varphi $ being concave with respect to the second variable we obtain Theorem \ref{existence1}. 
In the next section, we will prove that under $(SH_1)$ such $\varphi _1$ always exists
which makes Theorem \ref{non-existence-result-second-version} largely applicable.


For the proof we need to recall a number of properties satisfied by solutions to \eqref{problem}.
For $L=\Delta $ they were proved in \cite{ELMabrouk2004}, the general case is similar, see \cite{Ghardallou2}. 
\begin{lemma}[{ Lemma 5 in \cite{Ghardallou2}}]\label{comparaison-semi-elliptic}
	\vspace{0.2cm} Suppose that $\varphi $ satisfies ($H_2$). Let $u,v \in \mathcal{C}(\Omega)$ such
	that $Lu,Lv\in{\mathcal{L}}^1_{loc}(\Omega)$. If $$
	L u-\varphi(\cdot,u)\leq L v-\varphi(\cdot,v)$$
	in the sense of distributions and
	$$ \liminf\limits_{\underset{y\in\partial \Omega}{x\to y}}{(u-v)(x)}\geq0.
	$$
	Then:$$u-v\geq0\,\,in\,\,\Omega.$$
	
\end{lemma}
	
\medskip
For a bounded regular domain $D\subset \mathbb{R}^d$ and a non negative function $f$ 
continuous on $\partial D$, we define $U_D^{\varphi}f$ to be the function such that
$U_D^{\varphi}f=f$ on $\mathbb{R}^d\backslash D$ and $U_D^{\varphi}f| _{D}$ is the unique solution
of problem \begin{equation}
\left\{
\begin{array}{ll}
L u-\varphi(\cdot,u)=0, & \hbox{in $D$; in the sense of distributions;} \\
u\geq 0,               & \hbox{in $D$;} \\
u=f, & \hbox{on $\partial D$.}
\end{array}
\right. \label{1UDf}
\end{equation}
Existence of $U_D^{\varphi} f$ was proved in \cite{Ghardallou2} Theorem 4. Moreover,
\begin{equation}\label{identity}
H_Df=U^{\varphi }_Df+G_D\varphi (\cdot , U^{\varphi }_Df), \quad \mbox{in}\ D, \end{equation}
where $H_Df$ is a $L$-harmonic function in $D$ with boundary values $f$ and for a function $u$ and
\begin{equation}\label{greenpotential}
G_D(\varphi (\cdot , u))(x)=\int _{\Omega }G_{\Omega }(x,y)\varphi (y , u(y))\ dy.
\end{equation}
In particular $ U^{\varphi }_Df$ is not identically $0$ in $D$ {if it is so for $f$ on} 
$\partial D$.

Now we focus on properties of $U^{\varphi}_Df$. We say that $u$ is a supersolution to \eqref{problem} if $Lu-\varphi (\cdot ,u)\leq 0$ and a subsolution if $Lu-\varphi (\cdot ,u)\geq 0$. The following lemma is a direct consequence of Lemma \ref{comparaison-semi-elliptic} and existence of solutions to  \eqref{1UDf}. { For $L=\Delta $ is was proved in \cite{ELMabrouk2004}.}

 \begin{lemma}\label{properties-of-UD}
Suppose that $\varphi$ satisfies $(H_1), (H_2),(H_3)$ and let $D$ be a bounded regular domain such that $\bar D\subset \Omega $.
  $U_D^{\varphi}$ is monotone nondecreasing in the following sense
        \begin{equation}\label{monotone}
      U_D^{\varphi}f\leq U_D^{\varphi}g, \hbox{     if     } f\leq g \hbox{ in }\Omega.\end{equation}
 Let  $u$ be a continuous supersolution and $v$ a continuous subsolution of
      \eqref{problem} in $\Omega $. Suppose further that $D$ and $D'$ are regular bounded domains
      such that $D'\subset D \subset \Omega $. Then we have:

    \begin{equation}\label{super}
   U_D^{\varphi}u\leq u \ \hbox{     and     }\
            U_D^{\varphi}v\geq v.
\end{equation}
         \begin{equation}\label{super1}
  U_{D'}^{\varphi}u \geq U_D^{\varphi}u\ \hbox{and}\
            U_{D '}^{\varphi}v\leq U_D^{\varphi}v.\end{equation}
      If in addition, ($H_4$) holds \footnote{Notice that concavity with $(H_1)$ and $(H_2)$
           implies $(SH_1)$.} then $U_D^{\varphi}$ is convex function on
           $\mathcal{C}^+(\partial D)$ i.e. for every $\lambda\in[0,1]$ 
           
            \begin{equation}\label{concave}
      U_D^{\varphi}(\lambda f+(1-\lambda)g)\leq \lambda U_D^{\varphi}f+(1-\lambda)U_D^{\varphi}g. \end{equation}
In particular, for every $\alpha \geq 1 $
\begin{equation}\label{con1}
U_D^{\varphi}(\alpha f)\geq \alpha U_D^{\varphi}f.
\end{equation}

\end{lemma}

\medskip
Now, let $(D_n)$ be a sequence of bounded regular domains such that for every $ n\in \mathbb{N},$
$\overline{D_n}\subset D_{n+1} \subset \Omega$ and $\displaystyle\bigcup _{n=1}^{\infty }D_n=\Omega
$. Such a sequence will be called {\it a regular exhaustion} of $\Omega $ and it is used to generate solutions to \eqref{problem}. 

\begin{prop}[Proposition 10 in \cite{Ghardallou2}]\label{convergenceU_Dn}
	Let $g\in C^+(\Omega )$ be a $L$-superharmonic function. Then the sequence $(U_{D_n}^{\varphi}g)$ is decreasing to a solution
	$u\in\mathcal{C}^+(\Omega)$ of \eqref{problem} satisfying $u_g\leq g$.\footnote{Note here that $u_g$
		may be zero and usually and extra argument is needed to assure it is not.}

\end{prop}

Now we are ready to prove the main result of this section.

\begin{proof}(Proof of Theorem \ref{non-existence-result-second-version})
	Suppose that $Lu-\varphi_1(\cdot,u)=0$ has a nontrivial nonnegative bounded solution $\tilde u $ in
	$\Omega$. Let $(D_n)$ be an increasing sequence of bounded regular domain exhausting $\Omega$. Then, by Proposition \ref{convergenceU_Dn} for every $ \lambda\geq \lambda _1 =\|
	\tilde u\| _{L^{\infty }}>0$, $v_\lambda=\lim\limits_{n\to + \infty} U_{D_n}^{\varphi_1} \lambda$ is
	a nontrivial nonnegative bounded solution of $Lu-\varphi_1(\cdot,u)=0$ in $\Omega$ too.
	
	Let $\lambda\geq \lambda _1$. Then by Lemma \ref{properties-of-UD}    , $U_{D_n}^{\varphi_1} \lambda\geq \frac{\lambda}{\lambda _1} U_{D_n}^{\varphi_{1}}\lambda _1$. Therefore, letting $n\to \infty $ we obtain
	$$v_\lambda\geq \frac{\lambda}{\lambda _1} v_{\lambda _1}, $$
	where $v_{\lambda _1}= \lim\limits_{n\to \infty} U_{D_n}^{\varphi_1} \lambda _1$.
	
	 Furthermore, $\varphi\leq \varphi_1$ implies, by Lemma \ref{comparaison-semi-elliptic}, that
	$U_{D_n}^{\varphi}\lambda\geq U_{D_n}^{\varphi_1}\lambda$, because $U^{\varphi }_{D_n}\lambda $ is a supersolution to $Lu-\varphi _1(\cdot , u)$. Following this,
	$$u_{\lambda}=\lim _{n\to + \infty} U_{D_n}^{\varphi}\lambda     \geq v_{\lambda}.$$
	Suppose now that there is a large solution to \eqref{problem} denoted by $u$. \newline 
	Then $\liminf\limits_{x\to\partial \Omega}u=+\infty.$ 
	Hence for sufficiently large $n$, $u\geq U_{D_n}^{\varphi}\lambda$ on $\partial D_n$ and so, by Lemma \ref{comparaison-semi-elliptic}
	$$u\geq u_\lambda \geq v _{\lambda }.$$
	Consequently $u\geq \frac{\lambda}{\lambda _1} v_{\lambda _1}$ and so $ \frac{u}{\lambda}\geq \frac{1}{\lambda _1}v_{\lambda _1}$ for every $\lambda\geq \lambda _1$.
	When $\lambda $ tends to infinity, we get that $v_{\lambda _1}=0$ which gives a contradiction.
\end{proof}

\section{Domination by a concave function}\label{domin}
The aim of this section is to show that $(SH_1)$, $(H_2)$, $(H_3)$ imply existence of a function $\varphi _1$ concave with respect to the second variable and such that
$$\varphi (x,t)\leq  \varphi _1 (x,t), \quad \varphi _1 (x,0)=0.$$ 
Clearly, a nonnegative function $\psi $ concave on $[0,\infty )$, continuous at zero, $\psi (0)=0$ is dominated by an affine function. Indeed, given $\beta >0$, we have
$$
\psi (t)\leq \frac{t}{\beta }\psi (\beta ), \quad t\geq \beta $$
and so  
$$
\psi (t)\leq \frac{t}{\beta }\psi (\beta )+\sup _{0\leq s\leq \beta}\psi (s). $$
The idea behind $(SH_1)$ is to formulate as weak condition as possible to go beyond concavity in Theorem \ref{existence}. It turns out that $(SH_1)$ together with Theorem \ref{condominn} do the job. Clearly, the most delicate part is to guarantee that $\varphi _1(x,0)=0$.

\begin{theorem}\label{condominn}
	Suppose that $\varphi (x,t)$ satisfies $(SH_1)$, $(H_2)$, $(H_3)$. There is  $\varphi _1 (x,t)$ satisfying $(SH_1)$, $(H_2)$, $(H_3)$, $(H_4)$ such that
	$$\varphi (x,t)\leq  \varphi _1 (x,t).$$
	Moreover, { there exists a constant $C>0$ } such that
	$$\varphi _1(x,t)\leq C p(x)(t+1).$$
\end{theorem}

\begin{proof}
	For $t\geq 1$ $$\varphi (x,t)\leq 2p(x)t.$$ We need to dominate $\varphi $ for $t\leq 1$. 
	Let $\eta \in C^{\infty }(\mathbb{R} )$, $\eta \geq 0$, $supp \ \eta \subset (-1,1)$, $\eta (-t)=\eta (t)$ and $\int _{\mathbb{R} } \eta (s)\ ds=1$. 
	 Given $0<\delta \leq 1$, let $\eta _{\delta }(t)=\frac{1}{\delta } \eta (\frac{1}{\delta }t)$, $t\in \R$.
	Let $x\in \Omega$. We note  $\varphi _x(t)=\varphi (x,t)$, $t\in \R$. Then
	\begin{equation}
	\varphi _x * \eta _{\delta }(0)=\int _{-\de}^{\de }\varphi _x(-t) \eta _{\delta }(t)\ dt= \int_{-1}^1 \varphi(x, \de s) \eta(s) \;ds
	\end{equation}
	Hence
	\begin{equation}
	0\leq \inf _{\de } \varphi _x * \eta _{\delta }(0)=\lim _{\de \to 0} \varphi _x * \eta _{\delta }(0)=\varphi _x(0)=0.
	\end{equation}
	
	Secondly, $(\varphi _x * \eta _{\delta })'=\varphi _x * (\eta _{\delta })'$ and
	\begin{equation}
	(\eta _{\delta })'(t)=\frac{1}{\de ^2}\eta '\big (\frac{1}{\delta }t\big ).
	\end{equation}
	
	Moreover,
	\begin{align*}
	\int _{\mathbb{R}}|(\eta _{\delta })'(t)|\ dt &\leq \int _{\mathbb{R} }\frac{1}{\de ^2}|\eta '\big (\frac{1}{\delta }t\big )\ |dt \\
	&=\int _{\mathbb{R} }\frac{1}{\de }|\eta '(s)|\ ds .
	\end{align*}
	Therefore, if $0 \leq t\leq 2$ then 
	\begin{align*}
	|(\varphi _x * \eta _{\delta })'(t)|&\leq \int _{\mathbb{R}}\varphi _x(t-s)|(\eta _{\delta })'(s)|\ ds\\ &{\leq p(x)\frac{4}{\de } \int _{\mathbb{R} }|\eta '(s)|\ ds} .
	\end{align*}
		Consequently, there exists a constant $c_1$ such that for $0\leq t \leq 2$ we have
	\begin{equation}
	\varphi _x * \eta _{\delta }(t)\leq \frac{c_1}{\de }p(x)t+\varphi _x * \eta _{\delta }(0).
	\end{equation}
	Moreover,
	\begin{align*}
	\varphi _x * \eta _{\delta }(t)=&\int _{\mathbb{R}}\varphi _x(t-s)\eta _{\delta }(s)\ ds \\
	&\geq \int _{-\de }^0\varphi _x(t-s)\eta _{\delta }(s)\ ds \\
	&\geq \varphi _x (t)\int _{-\de }^0\eta _{\delta }(s)\ ds =\frac{1}{2} \varphi _x (t)
	\end{align*}
Hence $$\varphi _x(t)\leq  2\varphi_x *\eta _{\de }(t)$$
and so	
for $t\in [0,2]$ $$\varphi _x(t)\leq  \frac{2c_1}{\de }p(x)t+2\varphi _x * \eta _{\delta }(0)$$
Let $$\psi _{\delta }(x,t)=  \frac{2c_1}{\de }p(x)t+2\varphi _x * \eta _{\delta }(0)$$
and $$\psi(x,t)=\inf_{0<\de<1}\psi _{\delta }(x,t).$$
First we prove that for every fixed $x\in \Omega$, $\psi (x,t)$ is concave on $[0,2]$.
For $t ,s \in [0,2]$ and $ \al \in [0,1]$, we have
\begin{align*}
\psi (x ,\al t+(1-\al )s)=&\inf _{\delta } \psi _{\delta} (x,\al t+(1-\al )s)\\
=&\inf _{\delta}       (\al \psi _{\delta }(x,t)+(1-\al )\psi _{\delta }(x,s))
\end{align*}
and 
\begin{align*}
\inf _{\delta}       (\al \psi _{\delta }(x,t)+(1-\al )\psi _{\delta }(x,s))\geq &\inf _{\delta } (\al \psi_\de (x,t))    +     \inf _{\delta }   ((1-\al )\psi_\de (x,s))
\end{align*}
Hence $$\psi (x ,\al t+(1-\al )s) \geq \al \psi (x,t)+(1-\al )\psi (x,s) $$ and so $\psi (x,t)$ is continuous on $ (0,2 )$ in $t$.
Secondly,
$$
\psi (x,0)=\inf _{\delta}2\varphi _x * \eta _{\delta }(0)=2\varphi (x,0)=0$$
and for every $\de $,
\begin{align*}
\limsup _{t\to 0}\psi (x,t)\leq &\limsup _{t\to 0}(2c_1c(x)\frac{1}{\de }t+2\varphi _x * \eta _{\delta }(0))\\
\leq & 2\varphi _x * \eta _{\delta }(0)\leq 2\varphi _x(\de ).
\end{align*}
Hence $\lim\limits _{t\to 0^+}\psi (x,t)=0$ and so $\psi (x,t)$ is continuous on $\big [0,2\big )$.
Moreover, $\psi (x, \cdot )$ is nondecreasing and
\begin{align*}
\psi (x,t)&\leq \psi _1(x,t)\leq 2c_1p(x)t+2\varphi (x ,1)\\
&\leq 2c_1p(x)t+4p(x)\\
&\leq 4c_1p(x)(t+1).
\end{align*}
Finally, we define
$$ \varphi _1 (x,t)=2p(x)t+\psi (x,t)\quad \mbox{if}\quad 0\leq  t\leq 1, $$

$$ \varphi _1(x,t)=2p(x)t+\psi (x,1)\quad \mbox{if}\quad t> 1$$
and we set
$$ {\varphi _1 (x,t)=0 \quad \mbox{if}\quad t\leq  0.}$$
	
\end{proof}

\section{Large solutions to $L u-\varphi (\cdot ,u)=0$ under $(SH_1)$, $(H_2)$,  $ (H_3)$}\label{largesol1}

In this section we prove Theorem \ref{existence}. The argument is based on a very convenient characterization of existence of bounded solutions to \eqref{problem}. It is formulated in terms of thinness at infinity. 

Let $\Omega \subset \mathbb{R} ^d, d\geq 3$ be a domain. A subset $A\subset \Omega$
is called \emph{thin} at infinity if there is a continuous nonnegative $L$-superharmonic function
$s$ on $\Omega$ such that
$$ s\geq 1 \hbox{ on } A  $$and there is $ x_0\in
\Omega $ such that
$$ s(x_0)<1.$$

We say that $\Omega $ is {\it Greenian}
if there is a function $G_\Omega$ called {\it the Green function for }$L$ satisfying \begin{equation}\label{regularityGOmega}
G_{\Omega }(x,y)\in C^{\infty}\big ( \Omega \times \Omega \setminus \{ (x,x): x\in
\Omega \}\big )
\end{equation}  for every $y\in \Omega$
\begin{equation}\label{delta}
L G_{\Omega }(\cdot , y)=-\delta _y,\quad \mbox{in the sense of distributions}
\end{equation}
and
\begin{equation}\label{potential}
G_{\Omega }(\cdot , y),\quad \mbox{is a potential}
\end{equation}
i.e. every nonnegative $L $-harmonic function $h$ such that $h(x)\leq G_{\Omega }(x,y)$ is identically zero. For a given domain $\Omega $, the Green function $G_{\Omega }$ may or may not exist, but existence of $s$ as above implies that it does. 
\begin{theorem}(See Theorem 19 in \cite{Ghardallou2})\label{sufficient-and-necessary-condition-existence-solution}
	Suppose that $\Omega $ is Greenian and $\varphi:\Omega\times [0,\infty )\to [0,\infty )$ is measurable function satisfying
	$(H_1), (H_2),(H_3)$. Equation \eqref{problem} has a nonnegative nontrivial bounded solution in $\Omega$ if and
	only if there exists a Borel set $A\subset \Omega$
	which is thin at infinity and $c_0>0$ such that
	\begin{equation}\label{thinn}
	\int_{\Omega\backslash A} G_{\Omega}(\cdot,y)\varphi(y,c_0)\,dy\not\equiv \infty .
	\end{equation}
\end{theorem}
In the case of $L=\Delta $ and $\varphi (x,t)=p(x)t^{\gamma }$, $0< \gamma <1$, $p\in \mathcal{L}_{loc}^\infty$, Theorem \ref{sufficient-and-necessary-condition-existence-solution} was proved in \cite{ELMabrouk2004}. Notice that no concavity $(H_4)$ is required.

Moreover, in view of Theorems \ref{sufficient-and-necessary-condition-existence-solution} and \ref{condominn}, Theorem \ref{existence} is straightforward.

\begin{proof}[Proof of Theorem \ref{existence}]
	If $Lu- p(x) \psi(u)=0$ has a nonnegative nontrivial bounded solution then by Theorem \ref{sufficient-and-necessary-condition-existence-solution} there is a thin set $A\subset \Omega$ at infinty such that
	\begin{equation}
	\int _{\Omega \setminus A}G_{\Omega }(\cdot , y)p(y)\ dy \not\equiv \infty.
	\end{equation}
	Let $\varphi _1$ be the function constructed in Theorem \ref{condominn}. Then $\varphi _1$ can be taken such that
	$$
	\varphi _1(x,t)\leq Cp(x)(t+1)$$
	and so again by Theorem \ref{sufficient-and-necessary-condition-existence-solution}, $Lu-\varphi _1(\cdot, u)=0$ has a nonneagtive nontrivial bounded solution. Hence the conclusion follows by Theorem \ref{non-existence-result-second-version}.
\end{proof}

Now we are going to apply Theorem \ref{non-existence-result-second-version} to $\varphi $ that satisfies ($SH_1$).
\begin{theorem}\label{characterizationc}
	Let $\Omega$ be a Greenian domain. Assume that $\varphi $ satisfies $(SH_1)$, $(H_2)$, $(H_3)$ and there exists a thin set $A\subset \Omega $ at infinity such that the function $p(x)$ in $(SH_1)$ satisfies
	\begin{equation}\label{integrab}
	\int _{\Omega \setminus A}G_{\Omega }(\cdot , y)p(y)\ dy \not\equiv \infty .
	\end{equation}
	Then \eqref{problem} has a nonnegative nontrivial bounded solution and it has no large solution.
\end{theorem}
\begin{proof}
	By Theorem \ref{sufficient-and-necessary-condition-existence-solution}, there is a nonnegative nontrivial bounded solution to \eqref{problem}. Let $\varphi _1 (x,t)$ be the function constructed in Theorem \ref{condominn}. Then $$ \varphi _1 (x,t)\leq \ Cp(x)(t+1).$$ Hence there is a nonnegative nontrivial bounded solution to $Lu-\varphi _1(\cdot , u)=0$ and so by Theorem \ref{non-existence-result-second-version}, there is no large solution to \eqref{problem}.
\end{proof}

Suppose now that for every $t_0>0$ there is a constant $C_{t_0}>0$ such that for every $t\geq 0$ and $x\in \Omega$, $\varphi (x,t)\leq C_{t_0}\varphi (x,t_0)(t+1).$ We do not assume any integrability of $\varphi (x,t_0)$ in the spirit of \eqref{integrab}. Then

\begin{theorem}
Let $\Omega$ be a Greenian domain.	Assume that $\varphi $ satisfies $(H_1)$, $(H_2)$ and
	$(H_3)$. Suppose further that for every $t_0>0$ there is $C_{t_0}>0$ such that $$\varphi (x,t)\leq
	C_{t_0}\varphi (x, t_0)(t+1).$$ If \eqref{problem} has a nonnegative nontrivial bounded solution, then \eqref{problem} has no large solution.
\end{theorem}

\begin{proof}
	By Theorem \ref{sufficient-and-necessary-condition-existence-solution}, there exists a thin set $A\subset \Omega$ at infinity and $t_0>0$ such that
	\begin{equation}
	\int _{\Omega \setminus A}G_{\Omega }(\cdot , y)\varphi (y,t_0)\ dy \not \equiv\infty .
	\end{equation}
	Let $\varphi _1 (x,t)$ be the function constructed in Theorem \ref{condominn}. We can take $\varphi _1 $ such that $ \varphi _1 (x,t)\leq CC_{t_0}\varphi (x, t_0)(t+1)$. Then
	\begin{equation}
	\int _{\Omega \setminus A}G_{\Omega }(\cdot , y)\varphi _1(y,t_0)\ dy \not \equiv \infty .
	\end{equation}
	Hence there is a nonnegative nontrivial bounded solution to $Lu-\varphi _1(\cdot , u)=0$ and so by Theorem \ref{non-existence-result-second-version}, there is no large solution to \eqref{problem}.
\end{proof}

\section{Bounded solutions to $Lu-\varphi (\cdot,u)=0$}
Theorems \ref{condominn} and \ref{sufficient-and-necessary-condition-existence-solution} allow us to remove concavity from the following characterization of bounded solutions.
\begin{prop}\label{equivalent-proposition}
	Let $\Omega $ be a Greenian domain.
	Suppose that $\varphi (x,t)=p(x)\psi (t)$ satisfies
	$(SH_1)$, $(H_2)$ and $(H_3)$. Let $(D_n)$ be an increasing sequence of regular bounded domains exhausting $\Omega$. The
	following statements are equivalent:
	\begin{enumerate}
		\item  The equation \eqref{problem} has a nonnegative nontrivial bounded solution.
		\item  For every $c>0$, $\ds v_c=\inf _{n\in \NN }U_{D_n}^{\varphi}c$ is a nonnegative nontrivial bounded
		solution of \eqref{problem}
		\item  There exists $c>0$ such that $ \ds v_c=\inf _{n\in \NN }U_{D_n}^{\varphi}c$ is a nonnegative nontrivial
		bounded solution of \eqref{problem}.
	\end{enumerate}
	Furthermore if any of the above holds then
	\begin{equation}\label{max}
	\sup _{x\in \Omega }v_c(x)=c.
	\end{equation}
	
\end{prop}
{ The proof of proposition \ref{equivalent-proposition} is contained at the end of this section. We proceed as before: first we obtain the result for a concave nonlinear term i.e. under $(H_1)-(H_4)$ and then we apply Theorem \ref{condominn}.

\begin{prop}\label{equivalentcon}
	Suppose that $\varphi$ satisfies $(H_1)-(H_4)$. Then the statement of 
	Proposition \ref{equivalent-proposition} hold true.
\end{prop}
Proposition \ref{equivalentcon} was proved in \cite{ELMabrouk2004} for $L=\Delta $ and $\varphi (x,t)=p(x)t^{\gamma }$ where $0<\gamma <1$ and $p\in \mathcal{L}^{\infty}_{loc}$. Generalization to elliptic operators and $\varphi $ satisfying $(H_1)-(H_4)$ is straight forward and $\varphi $ does not need to be of the product form. }

\begin{proof}[Proof of Proposition \ref{equivalentcon}]
{	The proof is the same as in \cite{ELMabrouk2004} (Lemmas 3 and 4), but we include the argument here for the reader's convenience.
Let $u_n=U_{D_n}^{\varphi} c$ and $u_c=\inf _{n\in \N}u_n$. 
Under hypotheses $(H_1)-(H_4)$, if $u_c$ then
$\displaystyle\sup_{x\in \Omega}u(x)$ is either zero
or equal to $c$. Indeed, by Proposition \ref{convergenceU_Dn} $u_c$ is a nonnegative solution of \eqref{problem} bounded above by $c$. Suppose now that there exists $0 <c_0\leq c$ such that $\ds \sup_{x\in \Omega} u_c=c_0$.
		By Lemma \ref{comparaison-semi-elliptic} $$ U_{D_n}^\varphi (\dfrac{c}{c_0}u_c) \leq U_{D_n}^\varphi c= u_n.$$
		Also by Lemma \ref{properties-of-UD} $$ \dfrac{c}{c_0} U_{D_n}^\varphi u_c \leq  U_{D_n}^\varphi (\dfrac{c}{c_0} u_c).$$
		Hence $$U_{D_n}^\varphi u_c=u_c \leq  \dfrac{c_0}{c} u_n$$
and letting $n$ tend to infinity, we obtain $$u_c \leq  \dfrac{c_0}{c} u_c,$$
		which implies $c= c_0$.


	 Therefore with $(H_4)$, if any of conditions (1), (2), (3) is satisfied then \eqref{max} follows.
}
	It is clear that $(2) \Longrightarrow (3) \Longrightarrow (1).$ So it is enough to prove that (1)
	implies (2). Let $w$ be a nonnegative nontrivial bounded solution of \eqref{problem}.
	
	 Suppose first that $\ds r\geq \sup_{\Omega} w$. Then $\ds v=\lim _{n\to \infty}U_{D_n}^{\varphi}r$ is a nonnegative nontrivial bounded solution satisfying $w\leq v\leq r$ in $ \Omega$. Hence 
	\begin{equation}\label{sup}
	\sup _{x\in \Omega }v(x)=r.
	\end{equation}
	Secondly, we take $\ds 0 <c<\sup_{\Omega}w$. 

	By Lemma
	\ref{properties-of-UD}, $u_n=U_{D_n}^{\varphi}c\leq U_{D_n}^{\varphi}r=v_n$, in $D_n$. Then, we
	have
	$$G_{D_n}(\varphi(\cdot ,u_n))\leq G_{D_n}(\varphi(\cdot ,v_n)), \hbox{ in $D_n$}.$$
	Furthermore by \eqref{identity}
	$$ v_n+G_{D_n}(\varphi(\cdot,v_n))=r\ \hbox{      in $D_n$},$$
	and $$u_n+G_{D_n}(\varphi(\cdot,u_n))=c\ \hbox{      in  $D_n$}.$$ We can deduce $$0\leq c-u_n\leq
	r-v_n\ \hbox{        in $D_n$.}$$ When $n$ tends to infinity, we get $$c-u\leq r-v, \hbox{ in
		$\Omega$.}$$ Suppose now that $u$ is trivial. Then $$v\leq r-c\ \hbox{     in $\Omega$.}$$ But
	$\ds \sup_\Omega v=r$, which gives a contradiction.
\end{proof}

\begin{proof}[Proof of Proposition \ref{equivalent-proposition}]
	As before, it is enough to prove that (1) implies (2). By Theorem \ref{sufficient-and-necessary-condition-existence-solution}, there is a thin set $A\subset \Omega$ et infinity such that
	\begin{equation}
	\int _{\Omega \setminus A}G_{\Omega }(\cdot , y)p(y) \ dy \not \equiv \infty.
	\end{equation}
	Let $\varphi _1 (x,t)$ be the function constructed in Theorem \ref{condominn}. We can take $\varphi _1 $ such that $ \varphi _1 (x,t)\leq C p(x)(t+1)$ so again by Theorem \ref{sufficient-and-necessary-condition-existence-solution} $Lu-\varphi _1(\cdot , u)=0$ has a nonnegative nontrivial bounded solution. Let $c>0$. By Proposition \ref{equivalentcon} 
	$$
	v_c^1=\lim _{n\to \infty }U_{D_n}^{\varphi _1}c$$
	is a nonnegative nontrivial bounded solution of $Lu-\varphi _1(\cdot , u)=0$ and \begin{equation}
		\sup _{x\in \Omega }v_c^1(x)=c.
	\end{equation}
	But in view of Lemma \ref{comparaison-semi-elliptic}
	$$
c\geq 	v_c=\lim _{n\to \infty }U_{D_n}^{\varphi }c\geq \lim _{n\to \infty }U_{D_n}^{\varphi _1}c=v_c^1.$$
Then $v_c$	is a nonnegative nontrivial solution to \eqref{problem} satisfying
	$$
	\sup _{x\in \Omega }v_c(x)=c.$$
\end{proof}

\end{document}